\documentclass[11pt]{amsart}
\usepackage{amssymb}
\usepackage{graphicx}
\usepackage{amsmath}
\usepackage[dvipsnames]{xcolor}
\usepackage{comment}
\usepackage{hyperref}
 
\hypersetup{
    colorlinks=true,
    linkcolor=blue,
    filecolor=magenta,      
    urlcolor=cyan,
    citecolor=blue,
    }

\theoremstyle{plain}
\newtheorem{theorem}{Theorem}[section]

\newtheorem{lemma}[theorem]{Lemma}
\newtheorem{prop}[theorem]{Proposition}
\newtheorem{cor}[theorem]{Corollary}
\newtheorem{conj}[theorem]{Conjecture}
\newtheorem{claim}[theorem]{Claim}

\theoremstyle{definition}
\newtheorem{definition}[theorem]{Definition}
\newtheorem{remark}[theorem]{Remark}

\numberwithin{equation}{section}

\newcommand{\pol}{x\mathbb{Q}[x]}

\title{Sums, products, and exponents in two-colorings of the naturals}
\author{Ryan Alweiss}
\address{Department of Pure Mathematics and Mathematical Statistics and Trinity College, University of Cambridge, Cambridge, UK}
\email{ra699@cam.ac.uk}
\author{Matthew Bowen}
\address{Mathematical Institute, University of Oxford, Oxford, UK}
\email{ matthew.bowen@maths.ox.ac.uk}
\author{Marcin Sabok}
\address{Department of Mathematics and Statistics, McGill University, Montreal, Canada}
\email{marcin.sabok@mcgill.ca}
\date{}

\begin{document}

\maketitle

\begin{abstract}
    We prove that for any coloring of the naturals using two colors there are monochromatic sets of the form $\{x,y,xy,x+iy:i\leq k\}$ and $\{x,y,x^y,xy^i:i\leq k\}$ for any $k$.
\end{abstract}

\section{Introduction}

In this paper we are interested in the following well-known conjecture of Hindman.

\begin{conj}[\cite{hindman.conjecture}]\label{hindman conjecture}
    Any finite coloring of $\mathbb{N}$ contains monochromatic sets of the form $\{x,y,xy,x+y\}.$
\end{conj}

Despite its simplicity, Hindman's conjecture has remained recalcitrant for decades, and even special cases and relaxations of it have been the subject of much recent interest \cite{shkredov,cilleruelo2012combinatorial,green2016monochromatic,moreira2017monochromatic,bowen2025monochromatic,bowen.sabok,alweiss2022monochromatic,alweiss2023monochromatic,kousek2024revisiting,richter2025sums,green2025bounds}.  

Most relevant to the present paper, Moreira \cite{moreira2017monochromatic} has shown that any finite coloring of $\mathbb{N}$ contains monochromatic sets $\{x,xy,x+y\}$, i.e., Hindman's conjecture is true if we do not require that the step size $y$ is also the desired color, and the first named author has recently given an alternative proof \cite{alweiss2022monochromatic}.  The second author \cite{bowen2025monochromatic} has shown that Hindman's conjecture is true for colorings of $\mathbb{N}$ into two colors, the second and the third authors  have shown \cite{bowen.sabok} that Hindman's conjecture holds for arbitrary finite colorings of $\mathbb{Q}$ and the first author has shown \cite{alweiss2023monochromatic} that the general version of the Hindman conjecture with more than two variables also holds for arbitrary finite colorings of $\mathbb{Q}$.

In the present paper we return to the two-color case analyzed in \cite{bowen2025monochromatic}. In particular, we give a simpler proof of the main result from \cite{bowen2025monochromatic} that generalizes to deal with more complicated configurations.  Our main result is the following, which was previously only known in the case $k=1.$

\begin{theorem}\label{t1}
    For any $k\in \mathbb{N}$ and $2$-coloring of $\mathbb{N},$ there exist $x,y\in\mathbb{N}$ such that $\{x,y,xy, x+iy: i\leq k\}$ is monochromatic. 
\end{theorem}

In addition to proving Theorem \ref{t1}, a primary goal of this paper is to highlight that its basic proof strategy is fairly robust and allows for many alterations and adaptations.  Indeed, the second author has recently used a similar strategy to give a new proof and generalizations of the non-commuting Schur theorem for finite colorings of amenable groups \cite{bowen.non.commuting}.  The basic strategy used in the proof of Theorem \ref{t1} and the related results \cite{bowen2025monochromatic,bowen.sabok,bowen.non.commuting} is essentially as follows:

\begin{enumerate}
    \item First, show that the result holds when the color of the step size $y$ is ignored.  In this case, we are interested in finding monochromatic sets $\{x,xy,x+iy: i\leq k\},$ which can be done either through Moreira's work \cite{moreira2017monochromatic} or the first author's \cite{alweiss2022monochromatic}.

    \item Define an appropriate structure vs randomness dichotomy, and argue that a $2$-coloring must either be structured or random.

    \item Argue that in either case the extra information we obtain can be used to upgrade the proof from step (1) to control the color of the $y$ term as well. 
\end{enumerate}

In order to illustrate this technique we give two proofs of Theorem \ref{t1} based on the two ways of handling step (1) and using two different structure vs randomness dichotomies.  These are presented in the two subsections of Section \ref{section +}.

We also show that the same basic approach can be adapted to prove the following exponential version of the result.

\begin{theorem}\label{t2}
    For any $k\in \mathbb{N}$ and any $2$-coloring of $\mathbb{N},$ there exist $x,y\in\mathbb{N}$ with $\{x,y,x^y, xy^i: i\leq k\}$ monochromatic. 
\end{theorem}

The $k=1$ case of this result was proven by Sahasrabudhe \cite{sahasrabudhe2018monochromatic}, even for arbitrary finite colorings.  However, the result seems to be new for $k>1.$  Along the way, we also prove that any finite coloring of $\mathbb{N}$ contains monochromatic sets of the form $\{x,x^y,xy^i: i\leq k\},$ which gives a new exponential version of Moreira's theorem.

\section{Preliminaries}

Given a semigroup $(\Gamma,\cdot)$ and two subsets $S,T\subseteq \Gamma,$ we will denote $$S/T=\{g\in \Gamma: \exists t\in T \textnormal{ with } gt\in S\}.$$




We will only use the following definitions with the semigroups $(\mathbb{N},+)$ and $(\mathbb{N},\cdot).$

\begin{definition}

Let $(\Gamma,\cdot)$ be a semigroup.
    \begin{itemize}
        \item A subset $S\subseteq \Gamma$ is \textbf{syndetic} if there is a finite set $F\subset \Gamma$ such that $\Gamma=S/F.$

        \item A subset $T\subseteq \Gamma$ is \textbf{thick} if for every finite set $F\subset \Gamma$ there is a $t\in T$ such that $tF\subset T.$

        \item A subset $A\subset \Gamma$ is \textbf{piecewise syndetic} if there is a thick set $T$ and syndetic set $S$ such that $A=T\cap S.$

        \item Given a family $\mathcal{C}$ of subsets of $\Gamma$ we say that $A\subseteq \Gamma$ is in $\mathcal{C}^*$ if $A$ intersects every member of $\mathcal{C}$.

        \item Given a sequence $S$ of elements of $\Gamma$ we write $FS(S)$ for the set $\{a_{i_1}\cdot\ldots \cdot a_{i_n}:i_1<\ldots<i_n\in\mathbb{N}\}$,

        \item A subset of $\Gamma$ is \textbf{IP} if it contains a set $FS(S)$ for some infinite sequence $S$ of elements of $\Gamma$ and is called \textbf{IP}$_n$ if it contains a set $FS(S)$ for a sequence $S$ of elements of $\Gamma$ of length $n$.
    \end{itemize}

We will call subsets of $(\mathbb{N},+)$ \textbf{additively} syndetic, thick, etc., while subsets of $(\mathbb{N},\cdot)$ will be called \textbf{multiplicatively} syndetic, etc.
\end{definition}

\begin{remark}
    Note that a set is syndetic if and only if its complement is not thick.  Moreover, if we generate a finite coloring of $\mathbb{N}$ via fair dice rolls, then all color classes are both multiplicatively and additively thick with probability one.  For this reason, thick colorings can be viewed as `pseudorandom' colorings, and the complimentary case, where all color classes are syndetic, can be viewed as `pseudostructured.'  This is the main idea behind the proof strategies carried out in Sections \ref{sec:3.2} and \ref{exp}.
\end{remark}

We will make use of the following basic facts about these notions, whose proofs are well known and can be found, for example, in \cite{hindman.strauss}.
\begin{lemma}\label{basic}
\
   \begin{itemize}
       \item If $A$ is additively piecewise syndetic, then for any $n\in\mathbb{N}$ so are $nA$ and $A-n.$
       \item If $A$ is piecewise syndetic and $A=A_1\cup...\cup A_r,$ then there is an $i\in [r]$ such that $A_i$ is piecewise syndetic.
       \item If $F\subset \mathbb{N}$ is a finite set and $T$ is multiplicatively thick, then $\{T\in T: tF\subset T\}$ is multiplicatively thick.
       \item If $T$ is multiplicatively thick and $I$ is additively $IP_n^*$ for some $n,$ then $T\cap I$ is multiplicatively thick.
   \end{itemize}
\end{lemma}

We will also need the following version of van der Waerden's theorem, which is a consequence of the polynomial Hales-Jewett theorem.  The only difference with the more standard formulation is that we allow for polynomials with rational coefficients, but this makes no difference in the proof (see, e.g., \cite[Theorem 2.5]{bowen2025monochromatic}).

Below, we use the notation $\pol$ to denote the set of polynomials with rational coefficients and $0$ constant term.  


\begin{theorem}[$IP_n^*$ polynomial van der Waerden theorem]\label{vdw}
    For any additively piecewise syndetic set $A\subset \mathbb{N}$ and any finite $P\subseteq \pol$ there is an $n\in \mathbb{N}$ such that for any $S\subset \mathbb{N}$ with $|S|\geq n$ there is $d\in FS(S)$ such that $$A\cap\bigcap_{p\in P}(A-p(d))$$ is additively piecewise syndetic.
\end{theorem}

Also, we will use the following result of Moreira \cite{moreira2017monochromatic}. It should be noted that in \cite{moreira2017monochromatic} the result is only stated for integer polynomials but the same proof works for polynomials with rational coefficients.

\begin{theorem}\label{rational.moreira}
    Let $P\subseteq\pol$ be finite. For any finite coloring of $\mathbb{N}$ there exist $x,y\in\mathbb{N}$ and a color class $C$ such that $\{x,xy,x+p(y):p\in P\}\subseteq C$.
\end{theorem}

In Subsection \ref{sec:3.2} we will also make reference to the following notion of size.

\begin{definition}
    A subset $C\subseteq \mathbb{N}$ is \textbf{multiplicatively central} if there is a minimal idempotent ultrafilter $e\in(\beta\mathbb{N},\cdot)$ with $C\in e.$
\end{definition}

We will only need to make use of the following fact about multiplicatively central sets, which was proven in \cite[Theorem 3.1]{bowen2025monochromatic}.

\begin{theorem}\label{products of sums}
    If $C$ is multiplicatively central, then there is a sequence of subsets $(S_i)_{i\in \mathbb{N}}$ of $\mathbb{N}$ such that $|S_i|=i$ and for all finite $F\subseteq \bigcup_{i\in \mathbb{N}}S_i$ we have $$\prod_{i\in \mathbb{N}}\sum_{f\in S_i\cap F}f\subseteq C.$$
\end{theorem}


\section{The pattern $\{x,y,xy,x+iy: i\leq k\}$}\label{section +}

\subsection{A proof using distribution among polynomials}

\hspace{2mm}

In this subsection we give our first proof of Theorem \ref{t1}.  A basic outline of the argument is as follows.  Assume we are given a coloring $\mathbb{N}=R\cup B$. First, an application of the polynomial van der Waerden theorem will allow us to assume that there are arbitrarily long polynomial progressions in both color classes with the same step size, i.e., for any finite set of integral polynomials $P$ there is a $d\in\mathbb{N}$ and $a_0\in R$ and $a_1\in B$ such that $\{a_0+p(d): p\in P\}\subseteq R$ and $\{a_1+p(d): p\in P\}\subseteq B$.

From here there are two cases, depending on how the color classes are distributed among the set $P(d).$

If both color classes are well distributed in $P(d),$ then we will be able to exploit this fact to upgrade the first author's proof \cite{alweiss2022monochromatic} that the pattern $\{x, xy, x+p(y): p\in P\}$ is partition regular to also show that the color of the $y$ term can be controlled as well.

On the other hand, if the color classes are not well distributed in $P(d)$ then we will in fact be able to deduce that the set $P(d)$ is very close to being monochromatic in some color, say $R$.  Since the set $P(d)$ has a substantial amount of both additive and multiplicative structure, this will ensure that $R$ contains subsets of the form $\{x,y, xy, x+iy: i\leq k\}.$

\begin{proof}[First proof of Theorem \ref{t1}]

    We fix a two-coloring $\mathbb{N}=R \cup B$ and refer to these colors as red and blue.

    Our first goal is to show that we can assume that both color classes contain long polynomial progressions with a common step size.  Towards this, we will use the following claim together with the polynomial van der Waerden theorem.
    
   \begin{claim}\label{c1}
     If one of the colors is not additively piecewise syndetic, then for any finite set $P$ of integral polynomials without constant coefficient, there is a subset $\{x,y,xy,x+p(y):p\in P\}$ monochromatic in the other color.
   \end{claim}

   \begin{proof}
       Without loss of generality assume that $R$ is not additively piecewise syndetic. Then $B$ is additively thick.  Fix $y\in B$. Since $B$ is thick, we know $A=\{a\in B: \forall p\in P\ a+p(y)\in B\}$ is thick  as well and thus additively piecewise syndetic. Next, we know that $Ay$ is additively piecewise syndetic as well. So $Ay\cap R$ or $Ay\cap B$ is piecewise syndetic  and by assumption it is   $Ay\cap B$. Consequently, for any $x\in (Ay\cap B)/y$ we see that $\{x, y, xy, x+p(y):p\in P\}\subseteq B$.
   \end{proof}

   Consequently we may assume that both $R$ and $B$ are additively piecewise syndetic.  This allows us to deduce the following:
   
   \begin{claim}\label{common step size}
       For any finite set of polynomials $P\subset x\mathbb{Q}[x],$ there is $a_0\in R,$ $a_1\in B,$ and $d\in \mathbb{N}$ such that $\{a_0+p(d): p\in P\}\subseteq R$ and $\{a_1+p(d):p\in P\}\subseteq B.$ 
   \end{claim}

   \begin{proof}
       Since both $R$ and $B$ are additively piecewise syndetic, by the polynomial van der Waerden Theorem \ref{vdw}, we observe that $\{d: R\cap\bigcap_{p\in P}(R-p(d))\neq \emptyset\}$ is additively IP$^*$, and likewise for $B$.  Since the intersection of two IP$^*$ sets is non-empty, we can find a single $d$ such that both $R\cap\bigcap_{p\in P}(R-p(d))\neq \emptyset$ and $B\cap\bigcap_{p\in P}(B-p(d))\neq \emptyset.$  Choosing $a_0$ in the first set and $a_1$ in the second is as desired.
   \end{proof}

\begin{remark}
    We also have an alternative purely finitary argument that could take the place of Claims \ref{c1} and \ref{common step size}.  One just needs to find two polynomial progressions, $a_0+p(d)$ which is red and $a_1+p(d)$ which is blue.  Using the polynomial van der Waerden theorem, without loss of generality we can find red $a_0+p(d)$.  We claim that we can also ensure that $d$ is red, this is ``Brauer-Bergelson-Leibman".  If all polynomials $p(d)$ with small degree and small integer coefficients are blue, we are immediately done.  If some fairly small $d'=q(d)$ is red, we can rename $y \leftarrow d'=q(d)$, so $d$ is also red.  Then we are immediately done unless $a_0d+dp(d)$ are all blue.  Changing $a_1$ to $a_0d$ and renaming $y$ as $d^2$, we have $a_0+p(d)$ is red and $a_1+p(d)$ is blue as desired, for $p$ up to a certain size.
\end{remark}

Now suppose that we are proving Theorem \ref{t1} for a fixed choice of $k\in\mathbb{N}.$  Let $$P=\{\alpha x^j: \alpha\in \{\frac{a}{b}: a,b\leq 2k+1\}, j\leq K\}\subseteq x\mathbb{Q}[x],$$ where $K$ is a large constant depending on $k$.  Note that we do not require $\alpha$ to be an integer in the above definition, which is possible due to the version of van der Waerden's theorem, \ref{vdw}, that we are using. We will use Claim \ref{common step size} with this choice of $P$ and define $P(d)=\{p(d): p\in P\}$, where $d$ is the step size given by Claim \ref{common step size}.

   \begin{claim}\label{c2}
       Suppose that there are  polynomials $p,q\in P$ such that $p|q$ in $\mathbb{Q}[x]$ and $\{\frac{q}{p},pq,ip,iq: i\leq k\}\subseteq P$ and either: 
       
       \begin{enumerate}
           \item $p(d)\in R,$ $q(d)\in B,$ and $p(d)q(d)\in R$, or
           \item $p(d)\in B,$ $q(d)\in R,$ and $p(d)q(d)\in B$.
       \end{enumerate} 
       
       Then there is a monochromatic set of the form $\{x,y,xy, x+iy: i\leq k\}.$ 
   \end{claim}

   \begin{proof}

       Recall that we have chosen $a_0\in R$ and $a_1\in B$ via Claim \ref{common step size} so that $\{a_0+p(d):p\in P\}\subset R$ and $\{a_1+p(d):p\in P\}\subset B$.  We now assume that item (1) of the claim holds, as the other case is identical up to swapping the roles of $R$ and $B.$

       There are three cases to check:

       \textbf{Case 1}. There is an $i\leq k$ such that  $a_0\cdot p(d) + iq(d)\in R $. Then we can conclude by setting $x=a_0+iq(d)/p(d)$ and $y=p(d).$  
       
      \textbf{Case 2}.  Suppose $a_0\cdot p(d)q(d)\in R$.  Then we can conclude by setting $x=a_0$ and $y=p(d)q(d).$

       \textbf{Case 3}. Suppose that we are in neither Case 1 nor Case 2.  Then note that $a_0\cdot p(d) + iq(d)\in B$ for all $i\leq k$ and $a_0\cdot p(d)q(d)\in B.$  In this case, we can conclude by setting $x=a_0\cdot p(d)$ and $y=q(d).$ 
   \end{proof}




    \begin{claim}\label{c3}
        If the conditions of Claim \ref{c2} are not satisfied, then there is $e\in P(d)$ such that $\{ie^j: i\leq 2k+1, j\leq 2\}$ is monochromatic.
    \end{claim} 

    \begin{proof}
        To begin, consider the sequence of polynomials $d, d^2, d^3,\ldots,d^K\in P(d)$. We claim that if the hypothesis of Claim \ref{c2} is not satisfied, then for any $n<K$ either $\{d^i: i\leq n\}$ is monochromatic, or the set $\{d^i: n<i<K-n\}$ is.  To see this, suppose that $\{d^i: i\leq n\}$ is not monochromatic.  Then without loss of generality we may assume that $d\in R$ and there is some $i\leq n$ with $d^i\in B$.  Now note that if $d^{i+1}\in R,$ then Claim \ref{c2} condition (1) is satisfied with $p(d)=d$ and $q(d)=d^i$.  Consequently, we may assume that $d^{i+1}\in B,$ and repeating this logic ensures that $d^{j}\in B$ for all $j\geq i$.   
    
    Thus, after possibly replacing $d$ with $d^n$ for some $n$, we can assume that all polynomials $d, d^2, d^3,\ldots$ have the same color, say red. Now look at $2d, 2d^2, 2d^3,\ldots\in P(d)$.  Note that for any fixed $n$, if there is an $i\leq n$ with $2d^i\in B$ and Claim \ref{c2} condition (2) is not satisfied, then $\{2d^j: i\leq j\}\subseteq R.$  To see this, suppose towards a contradiction that $2d^j\in B$.  Then $p(d)=2q^i\in B$ and $q(d)=d^{j-i}\in R$ satisfy Claim \ref{c2} condition (2).  

    So, again after possibly replacing $d$ with $d^n,$ we may assume that $d,d^2,...$ and $2d,2d^2,...\in P(d)$ are both monochromatic in color red.  Repeating this $2k+1$ many times produces the desired monochromatic set with $e$ being the final $d$.
    \end{proof}

    Theorem \ref{t1} now follows from Claims \ref{c2} and \ref{c3}, where in the latter case we may take $x=e$ and $y=2e$.  
    
\end{proof}

\begin{remark}
    Note that the same argument can be adapted to show that, for any finite set of \textit{monomials} $P\subset x\mathbb{Q}[x],$ any $2$-coloring of $\mathbb{N}$ contains monochromatic sets $\{x,y,xy,x+p(y): p\in P\}.$  On the other hand, we do not know how to generalize this argument to deal with polynomials such as $p(x)= x^2+x$.
\end{remark}

\subsection{A proof using multiplicative structure}\label{sec:3.2}

\hspace{2mm}

Our next proof is based on the dichotomy between multiplicatively thick and syndetic sets and is a synthesis of the ideas used in the proof of  \cite[Theorem 3.4]{bowen2025monochromatic} and \cite[Proposition 3.1]{bowen.sabok}.  

\begin{proof}[Second proof of Theorem \ref{t1}]

Bergelson and Moreira \cite[Theorem 6.1]{bergelson2018measure} have shown that every multiplicatively syndetic subset of $\mathbb{N}$ contains subsets of the form $\{x+y,xy\}.$  Our starting point in this section is the following refinement of this fact.

\begin{prop}\label{s+c}
    Let $P$ be a finite set of integral polynomials without constant coefficients.  Any subset of $\mathbb{N}$ that is both multiplicatively syndetic and multiplicatively central contains subsets of the form $\{x,y, xy, x+p(y): p\in P\}$.
\end{prop}

Towards this, we will need the following refinement of Moreira's theorem \cite{moreira2017monochromatic}, whose proof is implicit in the proof of \cite[Theorem 3.4]{bowen2025monochromatic}.

\begin{prop}\label{central*}
    For any finite set of polynomials $P\subseteq\pol$ and any finite coloring of $\mathbb{N}$ the set  $$\{y: \exists x \textnormal{ with } \{x,xy,x+p(y):p\in P\} \textnormal{ monochromatic} \} $$

    is (multiplicatively central)$^*$.
\end{prop}

\begin{proof}
    Let $C\subseteq \mathbb{N}$ be a multiplicatively central set and suppose that we have fixed an $r$-coloring of $\mathbb{N}$ and a finite set $P$ of integral polynomials.  We fix a fast growing sequence $n_i$, depending only on the set $P$, to be specified later. By Theorem \ref{products of sums}, there is a sequence of sets $(I_i)_{i\in \mathbb{N}}$ such that $I_i$ is additively IP$_{n_i}$ and such that for any finite set $F\subset \mathbb{N}$ and elements $a_f\in I_f$ for $f\in F,$ we have $\prod_{f\in F}a_f\in C$.

    We now modify the proof of Moreira's theorem \cite{moreira2017monochromatic}, using Theorem \ref{vdw} to pick the step sizes from appropriate sets $I_f.$ More precisely, we inductively define a sequence of additively piecewise syndetic sets $A_1\supseteq A_2\supseteq\ldots$ integers $y_1,y_2,\ldots$ with $y_1\ldots y_{a-1}|y_a$ for every $a$ and finite sets of polynomials $P_1,P_2,\ldots$ such that

    \begin{itemize}
        \item $A_{i+1}\subseteq A_i\cap\bigcap_{p\in P_i}(A_i-p(y_i)).$

        \item $A_{i+1}\cdot y_1\cdot...\cdot y_i$ is monochromatic.

        \item $y_i\in I_n$ for some sufficiently large $n.$

        \item $P_1=P$ and $P_{i+1}=\{\frac{y_a\cdot\ldots\cdot y_b}{y_1\cdot...\cdot y_{a-1}}p: p\in P, a\leq b\leq i\}.$
    \end{itemize}
    At the $i$th step we use the fact that $I_i$ is IP$_{n_i}$ and use the $IP_n$ van der Waerden Theorem \ref{vdw} to ensure the first and third conditions are satisfied. This can be guaranteed if $n_i$ is chosen sufficiently large.  Finally, the second condition may be satisfied after passing to a subset by the partition regularity and dilation invariance of piecewise syndetic sets, Lemma \ref{basic}.

    We continue $r+1$ many steps. Write $A=A_{r+1}$.
    Now, there are $i<j$ such that the color of $A_r y_1\ldots y_i$ and $A_r y_1\ldots y_j$ are the same. Choose $x'\in A$ and put $x=x'y_1\ldots y_i$ and $y=y_{i+1}\ldots y_j$.
\end{proof}

\begin{proof}[Proof of Proposition \ref{s+c}]
    Suppose that $S$ is both multiplicatively syndetic and multiplicatively central. 
    By definition, there is a finite set $F\subset \mathbb{N}$ such that $\mathbb{N}\subset S/F.$  Let $P'=\{p/f: p\in P, f\in F \subseteq\pol \}$.  Consider the coloring $c:\mathbb{N}\rightarrow F$ where $c(n)=f$ if $nf\in S.$  Since $S$ is multiplicatively central, by Proposition \ref{central*} there is a $y\in S$ and $x'\in \mathbb{N}$ such that $\{x', x'y, x'+p'(y): p'\in P'\}$ is monochromatic under $c,$ say in color $f.$  Setting $x=x'f,$ we now see that $\{x,y,xy,x+p(y): p\in P\}\subseteq S.$
\end{proof}

Partition regularity of central sets now implies the following.

\begin{cor}\label{all_syndetic}
    Suppose that $\mathbb{N}=S_1\cup\ldots\cup S_r,$ where each $S_i$ is multiplicatively syndetic.  For any finite set $P$ of integral polynomials without constant coefficients, there is an $i\in [r]$ with $\{x,y,xy,x+p(y):p\in P\}\subseteq S_i.$
\end{cor}

Alternatively, if one is only interested in Corollary \ref{all_syndetic} rather than Proposition \ref{s+c}, then we can deduce it more directly by using the same argument from  \cite[Claim 4.1]{bowen.sabok}.  We repeat the details here for the reader's convenience.

\begin{proof}[Second proof of Corollary \ref{all_syndetic}]
    Fix a finite set $P\subseteq \pol.$ By definition, for each $i$ there is a finite set $F_i\subset \mathbb{N}$ such that $S_i/F_i\supseteq \mathbb{N}.$  Let $k=\max_{i\leq r}|F_i|$ and consider the coloring $c:\mathbb{N}\rightarrow [k]^r$ where an $n\in\mathbb{N}$ is colored based on the $r$-tuple listing the minimal $f_i\in F_i$ such that $f_in\in S_i.$  Let $P'=\{p/f: p\in P, f\in \bigcup_{i\leq r}F_i\}$.  By Theorem \ref{rational.moreira}, there are integers $x',y$ such that $S=\{x', x'y, x'+p'(y): p'\in P'\}$ is monochromatic according to the coloring $c.$  Now if $y\in S_i$, by the definition of $c$ there is an $f_i$ such that $f_iS\subseteq S_i.$  Consequently, setting $x=x'f_i$ is as desired. 
\end{proof}




Notice that the above proof works for any finite number of colors so long as all of the color classes are multiplicatively syndetic.  If we were able to prove a similar result assuming that all color classes are multiplicatively thick, then we could most likely combine these arguments using a strategy similar to that in \cite{bowen.sabok} to obtain a full proof of Conjecture \ref{hindman conjecture}.  Unfortunately, we are only able to show this for partitions into two thick sets.

\begin{lemma}\label{all thick}
    Suppose $k\in\mathbb{N}$ and $\mathbb{N}=R\cup B$ with both $R$ and $B$  multiplicatively thick. Then one of the colors contains a set $\{x,y,xy,x+iy:i\leq k\}$.
\end{lemma}
\begin{proof}
Suppose for the sake of contradiction that the lemma is false.
    Without loss of generality assume that $R$ is additively piecewise syndetic. Choose $N\in\mathbb{N}$ large enough depending on $k$.

    By induction, we construct a decreasing sequence $R_0,\ldots R_N\subseteq R$ of additively piecewise syndetic sets and a sequence $t_1,\ldots,t_N\in B$. 

    Along the way, we make sure that: 
    \begin{enumerate}
        \item\label{one}  $t_j\cdot\ldots\cdot t_i\in B$ for all $j\leq i$,
        \item\label{one.and.half} $\forall a\in R_i,\ a t_1\cdot\ldots \cdot t_i\in R$,
        \item\label{three} $\forall a\in R_{i+1},\ n<i,\ l\leq k, \ a  + \frac{lt_{n+1}\cdot\ldots\cdot t_{i+1}}{t_1\cdot\ldots t_n}\in R_i$.
    \end{enumerate}

    Before explaining the construction, let us explain why we can use it to produce a contradiction.  Let $a\in R_N$.  Since $R$ is multiplicatively thick, there is an $s\in R$ such that for all $i\leq N$ we have $t_1\cdot\ldots\cdot t_is\in R$ and  $t_1^2\cdot\ldots\cdot t_i^2s\in R$.  Note that this implies that for each $i\leq N$ there is $r_i\leq k$ such that $at_1\cdot\ldots\cdot t_i+r_it_1\cdot\ldots\cdot t_is\in B,$ or else we could set $x=at_1\cdot\ldots\cdot t_i\in R$ and $y=t_1\cdot\ldots\cdot t_is\in R$ and obtain a contradiction.  Since each $r_i\in [k],$ by the pigeonhole principle there are $i<j$ with $r=r_i=r_j$.  Set $x=at_1\cdot\ldots\cdot t_i+rt_1\cdot\ldots\cdot t_is$ and $y=t_{i+1}+\ldots t_j.$  We claim that $\{x,y,xy,x+ly: l\leq k\}\subseteq B$, which is a contradiction.  To see this, note that $x\in B$ by our choice of $r=r_i$ and $y\in B$ by condition (\ref{one}).  To see why $x+ly\in B,$ observe that $a+\frac{lt_{i+1}\cdot\ldots\cdot t_j}{t_1\cdot\ldots\cdot t_i}\in R_{j-1}\subseteq R_i$ by condition (\ref{three}), hence $x+ly\in B$ by (\ref{one.and.half}).  Finally, $xy=at_1\cdot\ldots\cdot t_j+rt_1\cdot\ldots\cdot t_js\in B$ by our choice of $r=r_j$. 

    Therefore, it suffices to construct the desired sequences satisfying (\ref{one})--(\ref{three}).  Towards this, assume that we have already constructed $R_0,\ldots,R_i$ and $t_1,\ldots,t_i$ satisfying (\ref{one})--(\ref{three}).  First, note that since $B$ is multiplicatively thick, the set $T=\{t:\forall j\leq i, \ t_j\cdot\ldots\cdot t_i\cdot t\in B\}$ is multiplicatively thick. For $t\in \mathbb{N}$ put
    $$R_t=\{a\in R_i: \forall j\leq i,\ l\leq k, \ a  + \frac{lt_{j+1}\cdot\ldots\cdot t_{i}t}{t_1\cdot\ldots t_j}\in R_i\}.$$
    Further, since $R_i$ is additively piecewise syndetic, by Theorem \ref{vdw}, we know 
   $$T'=\{t\in \mathbb{N}: R_t \mbox{ is additively piecewise syndetic}\}$$    
    is additively $IP_n^*$ for some $n.$  Consequently, by Lemma \ref{basic}, we know that $T\cap T'$ is multiplicatively thick.  Observe that if we choose $t_{i+1}\in T\cap T'$ then conditions (\ref{one}) and (\ref{three}) are satisfied.  
    
    We claim that there is $t_{i+1}\in T\cap T'$ such that there is an additively piecewise syndetic subset $R_{i+1}\subseteq R_t$ with $R_{i+1}t_1\cdot\ldots\cdot t_{i+1}\subset R$, so that condition (\ref{one.and.half}) is satisfied.  To see this, suppose for the sake of contradiction that for all $t\in T\cap T'$ and all piecewise syndetic subsets $R_t'\subset R_t,$ we have that $R_t't_1\cdot\ldots\cdot t_i\cdot t\cap R$ is not piecewise syndetic.  Choose any $t\in T\cap T'$ and a piecewise syndetic $R_{t}'\subset R_t$ such that $B'=R_t't_1\cdot\ldots 
    \cdot t_i\cdot t\subseteq B.$  Since $B'$ is also additively piecewise syndetic and $T'\cap T$ is multiplicatively thick, by Theorem \ref{vdw} and Lemma \ref{basic} there is  $t'\in T\cap T'$ such that $t\cdot t'\in T\cap T'$ and $B''=B'\cap\bigcap_{l\leq (t_1\cdot\ldots\cdot t_i\cdot t)^2k}(B'-lt')$ is piecewise syndetic.  Now observe that $B''/(t_1\cdot\ldots\cdot t_i\cdot t)\subseteq R_{tt'},$ and hence by assumption $B''t'\subseteq B.$  Choosing $x\in B''$ and $y=t',$ we see that $\{x,y,xy,x+ly:l\leq k\}\subseteq B,$ which contradicts our assumption that no such set was monochromatic.

\end{proof}

Finally, we can complete the second proof of Theorem \ref{t1}.  Note that in any two-coloring of $\mathbb{N}$ either both color classes are multiplicatively thick, or both colors are multiplicatively syndetic, or one color is both thick and syndetic.  We may now finish the proof by applying Lemma \ref{all thick} in the first case and Lemma \ref{s+c} in the latter two cases. 

\end{proof}

\section{The pattern $\{x,y, x^y, xy^i: i\leq k\}$}\label{exp}

In this section we prove Theorem \ref{t2}.  Although many of the details are different, our high level strategy is very similar to the one from our second proof of Theorem \ref{t1} in Subsection \ref{sec:3.2}.  In particular, we give separate proofs dealing with the case when all color classes are thick and the case when all classes are syndetic and show that they suffice to deduce Theorem \ref{t2}.  

Before beginning we first make the following reduction, which was implicit in Sahasrabudhe's proof of the exponential Schur theorem \cite{sahasrabudhe2018monochromatic}.

\begin{lemma}\label{log reduction}
    If any finite coloring of $\mathbb{N}$ contains a monochromatic set $\{x,y, x2^y, x+iy: i\leq k\}$ then any finite coloring contains a monochromatic set of the form $\{x,y,x^y,xy^i: i\leq k\}.$
\end{lemma}

\begin{proof}
    Given a coloring $c:\mathbb{N}\rightarrow [r]$ consider the coloring $c':\mathbb{N}\rightarrow [r]$ where $c'(n)=c(2^n).$  If $\{x', y', x'2^{y'}, x'+iy': i\leq k\}$ is $c'$-monochromatic then for $x=2^{x'}$ and $y=2^{y'}$ the set $\{x,y,x^y,xy^i: i\leq k\}$ is $c$-monochromatic.
\end{proof}

We now move on to the proof of Theorem \ref{t2}.

Recall from the previous section that Bergelson and Moreira \cite{bergelson2018measure} have shown that any multiplicatively syndetic set $S$ contains subsets of the form $\{x+y, xy\},$ and that Proposition \ref{s+c} strengthens this to ensure that if $S$ is also multiplicatively central then it contains the pattern $\{x,y,xy,x+y\}.$ Call a natural number $n$ \textit{multiplicatively odd} if the sum of the exponents in the prime factorization of $n$ is odd. Note that the multiplicatively odd numbers are multiplicatively syndetic and do not contain the pattern $\{x,y, xy\}.$  In particular, one cannot drop the assumption in Proposition \ref{s+c} that the set $S$ is central, and it is unclear what weakening of this assumption would still allow the result to hold.

In contrast, we now prove the following:

\begin{prop}\label{syndetic}
    If $S\subseteq \mathbb{N}$ is multiplicatively syndetic then for any finite $Q\subset\mathbb{Q}$ it contains subsets $\{x,y,x2^y,x+qy: q\in Q\}.$
\end{prop}

 Unlike the set $y$ from Proposition \ref{central*}, which was controlled by multiplicatively central sets, the set of step sizes $y$ giving monochromatic sets $\{x, x2^y, x+qy: q\in Q\}$ is well understood via purely additive structures.  This is made explicit in the following proposition, which is an exponential version of Moreira's theorem \cite{moreira2017monochromatic}.

 \begin{prop}\label{IP*}
     For any finite $Q\subseteq\mathbb{Q}$ and finite coloring of $\mathbb{N}$, the set $$\{y: \exists x \textnormal{ with } \{x,x2^y, x+qy:q\in Q\} \textnormal{ monochromatic}\}$$

     is additively IP$^*.$
 \end{prop}

 \begin{proof}
     Fix $r\in\mathbb{N},$ $Q\subseteq \mathbb{Q}$ finite, and suppose that $\mathbb{N}=C_1\cup...\cup C_r$ is a coloring.  Let $I$ be an IP set.

     Without loss of generality we may assume that $C_1$ is additively piecewise syndetic.  We now inductively define a sequence of additively piecewise syndetic sets $C_1=D_1\supseteq D_{1}'\supseteq D_2...\supseteq D_{r+1}$, integers $y_1,...,y_r$, and finite sets  $Q_1,...,Q_r\subseteq \mathbb{Q}$ such that

     \begin{enumerate}
         \item $D'_{i+1} = \bigcap_{q\in Q_i} (D_i-qy_i)$,
         \item $2^{y_1+...+y_i} D_{i+1}$ is monochromatic,  
         \item $Q_i=\{\frac{q}{2^{y_a+...+y_{b}}}: q\in Q \textnormal{ and }1\leq a\leq b\},$

         \item For all $1\leq a<b\leq r,$ $y_a+...+y_b\in I.$
     \end{enumerate}

     Such a sequences are constructed by iterative applications of the IP$^*$ van der Waerden theorem \ref{vdw} and Lemma \ref{basic} and the pigeonhole principle.

     By the pigeonhole principle and Condition (2) of the construction, there must be $i<j$ such that $2^{y_1+...+y_i}D_{i+1}$ and $2^{y_1+...+y_j}D_{j+1}$ are monochromatic in the same color.  Set $y=y_{i+1}+...+y_j,$ which is in $I$ by Condition (4).  Now choose $x'\in D_{j+1}$ and set $x=x'\cdot 2^{y_1+...+y_i}\in 2^{y_1+...+y_i}D_{i+1}$.  As $D_{j+1}\subseteq D'_a\subseteq D_{i+1}$ for all $i+1\leq a\leq j,$ by Conditions (1) and (3) we have that $x'+\frac{\sum_{a=i+1}^jqy_a}{2^{y_1+...+y_i}}\in D_{i+1}$ for all $q\in Q.$  In particular, we have $x+qy\in 2^{y_1+...+y_i}D_{i+1}.$  Finally, by Condition (2) we see that $\{x,x2^y, x+qy: q\in Q\}$ is monochromatic as desired. \end{proof}

 Proposition \ref{syndetic} now follows in much the same way as Proposition \ref{s+c} from the previous section.

\begin{proof}[Proof of Proposition \ref{syndetic}]
    Fix a finite $Q \subset \mathbb{Q}$ and let $S$ be a multiplicatively syndetic set.

    Recall that Hindman's theorem implies that any finite coloring of $\mathbb{N}$ contains a monochromatic additive $IP$ set.  Since additive $IP$ sets are dilation invariant, this implies that multiplicatively syndetic sets are additively $IP.$

    Since $S$ is multiplicatively syndetic, by definition there is a finite set $F\subseteq \mathbb{N}$ such that $\mathbb{N}\subseteq S/F.$  Consider the coloring $c:\mathbb{N}\rightarrow F$ where $c(n)=f$ if $f\in F$ is the minimal number with $nf\in S.$  Set $Q'=\{q/f: q\in Q \textnormal{ and }f\in F\}.$

    Since multiplicatively syndetic sets are additively $IP$, by Proposition \ref{IP*}, there is a $y\in S$ and $x'\in \mathbb{N}$ such that the set $\{x',x'2^y, x'+qy: q\in Q'\}$ is monochromatic according to $c,$ say in color $f\in F.$  Now notice that for $x=x'f$ we get $\{x,y,x2^y,x+qy: q\in Q\}\subseteq S$.  \end{proof}




   With the syndetic case taken care of, we are in a position to finish the proof of Theorem \ref{t2}.

\begin{proof}[Proof of Theorem \ref{t2}] 

Suppose that we are proving Theorem \ref{t2} for a fixed $k\in \mathbb{N}$ and coloring $\mathbb{N}=R\cup B$.  By Lemma \ref{log reduction}, it suffices to find a monochromatic set $\{x,y,x2^y,x+iy: i\leq k\}$.  By Proposition \ref{syndetic}, we can assume both color classes are multiplicatively thick, and by the same argument as in Claim \ref{c1} we may assume that both colors are additively piecewise syndetic.

    Iteratively apply the $IP_n^*$ van der Waerden theorem \ref{vdw} three times to find two monochromatic progressions of the form 
    $$P_1=a+Q_1d_1+Q_2d_2+Q_3d_3\subseteq R,$$

    $$P_2=b+Q_1d_1+Q_2d_2+Q_3d_3\subseteq B,$$
where $Q_1\subseteq Q_2\subseteq Q_3\subseteq \mathbb{Q}$ are large finite sets and $d_i$ are chosen by thickness to ensure $Q_1d_1\subseteq R,$ $Q_2d_2\subseteq B,$ and $Q_3d_3\subseteq R$.  Here, $Q_1=[k^2]$, $Q_2=\{\frac{i}{j}:i,j\leq k^22^{kd_1}\},$ and $Q_3=\{\frac{i}{j}: i,j\leq k^22^{k(d_1+d_2)}\}$ suffice for our purposes. 

\vspace{2mm}

\textbf{Case 1}. If  $q_1d_1+q_2d_2\in R$ for some choice of $q_1\in [k]$ and $q_2\in [k2^{kd_1}]$, then consider the sub-progression $$\{a, a+i(q_1d_1), a+i(q_1d_1+q_2d_2), (a+j\frac{q_2d_2}{2^{q_1d_1}})+iq_1d_1: i,j\leq k\}\subseteq P_1\subseteq R.$$  Recall that $q_1d_1\in R$ by our choice of $d_1$.  Thus, we can assume that $a2^{q_1d_1}\in B$ or else we could conclude by setting $x=a$ and $y=q_1d_1$ and considering the above sub-progression.  Similarly, we can conclude $(a+j\frac{q_2d_2}{2^{q_1d_1}})2^{q_1d_1}\in B$.  Finally, since $q_1d_1+q_2d_2\in R$ by our assumption, we can also assume that $a2^{q_1d_1+q_2d_2}\in B$.   But now note that for $x=a2^{q_1d_1}$ and $y=q_2d_2$ that $\{x,y,x2^y,x+iy:i\leq k\}\subseteq B$.

\vspace{2mm}

The remaining cases now follow from a very similar analysis.

\vspace{2mm}

\textbf{Case 2}. Assume $q_1d_1+q_2d_2\in B$ for all $q_1\in [k]$ and $q_2\in [k2^{kd_1}]$. 

\vspace{2mm}

\mbox{}\indent \textbf{Subcase 1}. If $(q_1d_1+q_2d_2)+q_3d_3\in B$ for some choice of $q_3\in [k2^{k(d_1+d_2)}]$, then the same argument from Case 1 with $(q_1d_1+q_2d_2)$ taking the place of $q_1d_1,$ $q_3d_3$ taking the place of $q_2d_2,$ and $P_2$ taking the place of $P_1$ produces a monochromatic set $\{x,y,x2^y, x+iy: i\leq k\}$.

\vspace{2mm}

\mbox{}\indent\textbf{Subcase 2}. Assume $(q_1d_1+q_2d_2)+q_3d_3\in R$ for all $q_1\in [k]$ and $q_2\in [k2^{kd_1}],$ and $q_3\in [k2^{k(d_1+d_2})]$. 

\vspace{2mm}

\mbox{}\indent\indent\textbf{Subsubcase 1}. If $q_2 d_2+q_3 d_3\in B$, for some choice of $q_2\in [k2^{kd_1}]$ and $q_3\in [k2^{k(d_1+d_2)}]$ then proceed  as in Case 1, with $q_2d_2$ taking the place of $q_1d_1$, $q_3d_3$ taking the place of $q_2d_2,$ and $P_2$ taking the place of $P_1$.

\vspace{2mm}

\mbox{}\indent\indent\textbf{Subsubcase 2}. Assume $q_2 d_2+q_3 d_3\in R$ for all $q_2\in [k2^{kd_1}]$ and $q_3\in [k2^{k(d_1+d_2)}]$.
 Then for $x=d_2+d_3\in R$ and $y=d_1\in R$, this ensures $x+iy=(id_1+d_2)+d_3\in R$ by the assumption of Subcase 2 and $x2^y=2^{d_1}(d_2+d_3)=2^{d_1} d_2+2^{d_1} d_3\in R$, by the assumption of current Subsubcase 2, as desired.
\end{proof}

\bibliographystyle{amsalpha} 
\bibliography{bib} 

\end{document}